\documentclass[11pt,a4paper]{article}

\usepackage{inputenc}
\usepackage{amsmath}
\usepackage{bm}
\usepackage{bbold}
\usepackage{amsthm}
\usepackage{enumerate}

\usepackage{hyperref}

\title{Explicit Solution of a Tropical Optimization Problem with Application to Project Scheduling\thanks{Mathematical Methods and Optimization Techniques in Engineering: Proc. 1st Intern. Conf. on Optimization Techniques in Engineering (OTENG '13), WSEAS Press, 2013, pp.~39--45.}}

\author{Nikolai Krivulin\thanks{Faculty of Mathematics and Mechanics, Saint Petersburg State University, 28 Universitetsky Ave., Saint Petersburg, 198504, Russia, 
nkk\textless at\textgreater math.spbu.ru.}\thanks{The work was supported in part by the Russian Foundation for Humanities under Grant \#13-02-00338.}
}

\date{}

\newtheorem{theorem}{Theorem}
\newtheorem{lemma}[theorem]{Lemma}
\newtheorem{corollary}[theorem]{Corollary}

\begin{document}

\maketitle

\begin{abstract}
A new multidimensional optimization problem is considered in the tropical mathematics setting. The problem is to minimize a nonlinear function defined on a finite-dimensional semimodule over an idempotent semifield and given by a conjugate transposition operator. A special case of the problem, which arises in just-in-time scheduling, serves as a motivation for the study. To solve the general problem, we derive a sharp lower bound for the objective function and then find vectors that yield the bound. Under general conditions, an explicit solution is obtained in a compact vector form. This  result is applied to provide new solutions for scheduling problems under consideration. To illustrate, numerical examples are also presented.
\\

\textbf{Key-Words:} idempotent semifield, tropical optimization problem, nonlinear objective function, span seminorm, project scheduling.
\\

\textbf{MSC (2010):} 65K10, 15A80, 65K05, 90C08, 90B35
\end{abstract}

\section{Introduction}

Tropical (idempotent) mathematics is concerned with the theory and applications of semirings with idempotent addition. Tropical mathematics had its origin in seminal works \cite{Pandit1961Anew,CuninghameGreen1962Describing,Giffler1963Scheduling,Vorobjev1963Theextremal,Romanovskii1964Asymptotic}, which introduced it as a constructive tool to represent and solve real-world problems in operations research, such as scheduling problems that was examined in \cite{CuninghameGreen1962Describing,Giffler1963Scheduling}. Over the past few decades, significant progress has been achieved in the field, which is reflected in several monographs (see, e.g., \cite{Baccelli1993Synchronization,Kolokoltsov1997Idempotent,Golan2003Semirings,Heidergott2006Maxplus,Gondran2008Graphs,Krivulin2009Methods,Butkovic2010Maxlinear} for recent publications) and in a wide range of research papers.

Since early studies \cite{Cuninghamegreen1976Projections,Superville1978Various}, optimization problems that are formulated and solved within the framework of tropical mathematics have constituted an important research domain in the field. The problems are to minimize or maximize functions defined on finite-dimensional semimodules over idempotent semifields subject to constraints in the form of linear equalities and inequalities. Both linear and nonlinear objective functions are considered.

The span (range) seminorm, which is defined as the maximum deviation between components of a vector, is one of the objective functions that are encountered in the problems. This function is used as an optimality criterion for some problems in a range of areas from the analysis of Markov decision processes \cite{Bather1973Optimal,Puterman2005Markov} to the form-error measurement in precision metrology \cite{Murthy1980Minimum,Gosavi2012Form}. In the context of tropical mathematics, the span seminorm has been introduced by \cite{Cuninghamegreen1979Minimax,CuninghameGreen2004Bases}, where it was called the range seminorm. 

The span seminorm appeared in \cite{Butkovic2009Onsome,Kin2010Optimizing} in a tropical optimization problem drawn from machine scheduling. A manufacturing system is considered, in which machines start and finish under certain precedence constraints to produce components for final products. The problem is to find the starting times of each machine so that the completion times are spread over a shortest possible period of time. A solution to the problem is given in a somewhat complicated form that involves two reciprocally dual idempotent semifields.  

In this paper, we examine a more general problem of just-in-time scheduling \cite{Demeulemeester2002Project,Tkindt2006Multicriteria}. The problem is formulated in the common setting of project scheduling in terms of activities that are conducted under various precedence relations between their initiation and completion times. The goal is to design a schedule that provides, as far as possible, a single common completion time of all activities and can thus be solved by minimizing a span seminorm. Compared to that in \cite{Butkovic2009Onsome,Kin2010Optimizing}, the new problem takes into account additional constraints that limit the time intervals between initiation of activities.

We represent the precedence relations by linear vector equalities and inequalities in an idempotent semifield. The span seminorm is written in a straightforward, if not linear, vector form. As a result, we arrive at a constrained optimization problem with a nonlinear objective function, which involves a conjugate transposition operator, subject to linear constraints. 

The above mentioned problem serves as a motivating example and a starting point to define and solve a new general tropical optimization problem in a rather formal setting. We exploit the fact that the application of the solution of linear inequalities in \cite{Krivulin2006Solution,Krivulin2009Methods} reduces the problem to an unconstrained problem with new variables. We examine an extended version of the unconstrained problem that is formulated in terms of a general idempotent semifield. To solve the latter problem, the solution approach developed in \cite{Krivulin2005Evaluation,Krivulin2009Methods,Krivulin2012Anew} is used based on the derivation of a sharp lower bound for the objective function and then the construction of vectors that give the bound. 

We obtain a direct solution to the extended problem under fairly general conditions and represent it in a compact vector form in terms of the carrier semifield. Then, the above mentioned scheduling problems are solved as particular cases. Specifically, a new solution to the machine scheduling problem examined in \cite{Butkovic2009Onsome,Kin2010Optimizing} is obtained as a consequence.

The solutions are given in an explicit form that is suitable for both formal analysis and practical implementation. The results obtained, which are first aimed at formulating and solving new tropical optimization problems, may also serve as a contribution to project scheduling, which offers direct solutions rather than indirect solutions to many scheduling problems that can often be solved only by sophisticated computational algorithms \cite{Demeulemeester2002Project,Tkindt2006Multicriteria}.

The paper is organized as follows. It begins with a motivating problem drawn from just-in-time scheduling in Section~\ref{S-ME}. Furthermore, we give a brief introduction to basic definitions, notation, and preliminary results in tropical mathematics in Section~\ref{S-PR} to provide a formal framework for subsequent results. Section~\ref{S-OP} suggests the main results that include the definition of and a solution to a general optimization problems with nonlinear objective functions. Application of the results to optimal scheduling problems are presented and illustrated with numerical examples in Section~\ref{S-AOS}.

\section{Motivating Example}\label{S-ME}

We start with a real-world problem that is drawn from project scheduling and intended to both motivate and illustrate further results. The problem arises in just-in-time manufacturing and aims to design a schedule that minimizes the maximum deviation between the completion times of the  activities in a project subject to various activity precedence constraints. For more details and references on project scheduling, and specifically on just-in-time scheduling, one can consult \cite{Demeulemeester2002Project,Tkindt2006Multicriteria}. 

Consider a project which consists of $n$ activities (jobs, tasks) that operate under start-finish and start-start precedence constraints. The start-finish constraints require that a minimal time lag be held between the initiation of one activity and the completion of another. Each activity is assumed to be completed as early as possible to meet these constraints. The start-start constraints specify a minimal time lag between the initiation of any two activities. The problem is to find a schedule that provides, as far as possible under the constraints, a single common completion time for all activities.

For each activity $i=1,\ldots,n$, let $x_{i}$ be the initiation time, $y_{i}$ be the completion time, and $c_{ij}$ be the minimum possible time lag between the initiation of activity $j=1,\ldots,n$ and the completion of activity $i$. Given $c_{ij}$, the completion time of activity $i$ must satisfy the start-finish precedence relations
$$
x_{j}+c_{ij}\leq y_{i},
\quad
j=1,\ldots,n,
$$ 
with at least one inequality holding as equality. Note that we assume $c_{ii}\geq0$ for all $i$. Provided that $c_{ij}$ is not given for some $j$, we put $c_{ij}=-\infty$.

Now we combine the relations into one equality of the form
$$
\max_{1\leq j\leq n}(x_{j}+c_{ij})
=
y_{i}.
$$

Furthermore, let $d_{ij}$ be the minimum possible time lag between the initiation of activity $j$ and the initiation of activity $i$. Once again, we assume $d_{ij}=-\infty$ if no lag is specified for $i$ and $j$. Due to the start-start constraints, we have relations
$$
x_{j}+d_{ij}\leq x_{i},
\quad
j=1,\ldots,n,
$$ 
and rewrite them as one inequality
$$
\max_{1\leq j\leq n}(x_{j}+d_{ij})
\leq
x_{i}.
$$

We define an objective function for the optimal scheduling problem under study. We take the maximum deviation between the completion times as a criterion, which is equal to zero only when a schedule provides a single common completion time for all activities. The criterion has the form of the span seminorm:
$$
\max_{1\leq i\leq n}y_{i}
-
\min_{1\leq i\leq n}y_{i}
=
\max_{1\leq i\leq n}y_{i}
+
\max_{1\leq i\leq n}(-y_{i}).
$$

We now formulate an optimization problem of interest. Given $c_{ij}$ and $d_{ij}$ for all $i,j=1,\ldots,n$, the problem is to find $x_{1},\ldots,x_{n}$ such that
\begin{equation}
\begin{aligned}
&
\text{minimize}
&&
\max_{1\leq i\leq n}y_{i}
+
\max_{1\leq i\leq n}(-y_{i}),
\\
&
\text{subject to}
&&
\max_{1\leq j\leq n}(x_{j}+c_{ij})
=
y_{i},
\\
&
&&
\max_{1\leq j\leq n}(x_{j}+d_{ij})
\leq
x_{i},
\quad
i=1,\ldots,n.
\end{aligned}
\label{P-yyxayxbxI}
\end{equation}

Below, we represent the problem in the tropical mathematics setting and then solve it directly in a compact vector form.

\section{Preliminary results}\label{S-PR}

In this section, we give a brief overview of the main algebraic definitions, notation and preliminary results, which provide a basis for the subsequent solution to tropical optimization problems and applications to project scheduling. Both concise introductions to and thorough presentation of tropical mathematics are presented in various forms in a range of works, including \cite{Baccelli1993Synchronization,Cuninghamegreen1994Minimax,Kolokoltsov1997Idempotent,Golan2003Semirings,Gondran2008Graphs,Heidergott2006Maxplus,Akian2007Maxplus,Litvinov2007Themaslov,Butkovic2010Maxlinear}. Below, we mainly adhere to the results in \cite{Krivulin2006Solution,Krivulin2009Methods}, which offer a useful framework to obtain direct solutions in a compact form. For additional details, one can consult other publications listed above.

\subsection{Idempotent Semifield}

Let $\mathbb{X}$ be a set that is closed under two associative and commutative operations, addition $\oplus$ and multiplication $\otimes$, and equipped with their neutral elements, zero $\mathbb{0}$ and identity $\mathbb{1}$. Addition is idempotent, which means that $x\oplus x=x$ for all $x\in\mathbb{X}$. Multiplication is distributive over addition and invertible, which implies that each $x\in\mathbb{X}_{+}$, where $\mathbb{X}_{+}=\mathbb{X}\setminus\{\mathbb{0}\}$, has an inverse $x^{-1}$ to satisfy $x^{-1}\otimes x=\mathbb{1}$. Since $\mathbb{X}_{+}$ forms a group under multiplication, the structure $\langle\mathbb{X},\mathbb{0},\mathbb{1},\oplus,\otimes\rangle$ is commonly referred to as the idempotent semifield.

The integer power is introduced as usual. For any $x\in\mathbb{X}_{+}$ and integer $p>0$, we have $x^{0}=\mathbb{1}$, $\mathbb{0}^{p}=\mathbb{0}$, $x^{p}=x^{p-1}\otimes x$, and $x^{-p}=(x^{-1})^{p}$.

In what follows, the multiplication sign $\otimes$ is dropped for simplicity. The power notation is used in the sense of the above mentioned definition.

The idempotent addition produces a partial order, by which $x\leq y$ if and only if $x\oplus y=y$. The partial order is assumed to extend to a consistent total order over $\mathbb{X}$. The relation symbols and the minimization problems are thought in the context of this order for here on.

As examples of the general semifield under consideration, one can take
\begin{align*}
\mathbb{R}_{\max,+}
&=
\langle\mathbb{R}\cup\{-\infty\},-\infty,0,\max,+\rangle,
\\
\mathbb{R}_{\min,+}
&=
\langle\mathbb{R}\cup\{+\infty\},+\infty,0,\min,+\rangle,
\\
\mathbb{R}_{\max,\times}
&=
\langle\mathbb{R}_{+}\cup\{0\},0,1,\max,\times\rangle,
\\
\mathbb{R}_{\min,\times}
&=
\langle\mathbb{R}_{+}\cup\{+\infty\},+\infty,1,\min,\times\rangle,
\end{align*}
where $\mathbb{R}$ is the set of reals and $\mathbb{R}_{+}=\{x\in\mathbb{R}|x>0\}$. 

Specifically, the semifield $\mathbb{R}_{\max,+}$ has the null $\mathbb{0}=-\infty$ and identity $\mathbb{1}=0$. Each $x\in\mathbb{R}$ has its inverse $x^{-1}$ given by $-x$ in standard notation. For any $x,y\in\mathbb{R}$, the power $x^{y}$ is equal to the arithmetic product $xy$. The order, which is induced by addition, corresponds to the natural linear order on $\mathbb{R}$.

\subsection{Matrix Algebra}

We now consider matrices over $\mathbb{X}$ and denote the set of matrices with $m$ rows and $n$ columns $\mathbb{X}^{m\times n}$. A matrix with all entries equal to $\mathbb{0}$ is called the zero matrix. A matrix is row (column) regular, if it has no zero rows (columns). A matrix is regular, if it is both row and column regular.

For any matrices $\bm{A}=(a_{ij})$, $\bm{B}=(b_{ij})$, and $\bm{C}=(c_{ij})$ of appropriate dimensions, and a scalar $x$, matrix addition, matrix and scalar multiplication are routinely defined as
\begin{gather*}
\{\bm{A}\oplus\bm{B}\}_{ij}
=
a_{ij}\oplus b_{ij},
\qquad
\{\bm{B}\bm{C}\}_{ij}
=
\bigoplus_{k}b_{ik}c_{kj},
\\
\{x\bm{A}\}_{ij}
=
xa_{ij}.
\end{gather*}

For any matrix $\bm{A}$, its transpose is denoted $\bm{A}^{T}$.


Consider square matrices in $\mathbb{X}^{n\times n}$. A matrix that has all diagonal entries equal to $\mathbb{1}$ and off-diagonal entries equal to $\mathbb{0}$ is the identity matrix represented by $\bm{I}$. For any matrix $\bm{A}$, the trace is given by
$$
\mathop\mathrm{tr}\bm{A}
=
\bigoplus_{i=1}^{n}a_{ii}.
$$ 

The matrices with only one column (row) are routinely referred to as column (row) vectors. We denote the set of column vectors of order $n$ by $\mathbb{X}^{n}$.

A vector that has all components equal to $\mathbb{0}$ is the zero vector. A vector is regular if it has no zero components.

Let $\bm{x}$ be a regular column vector and $\bm{A}$ be a matrix. It is not difficult to see that the vector $\bm{A}\bm{x}$ is regular only when the matrix $\bm{A}$ is row regular. Similarly, the row vector $\bm{x}^{T}\bm{A}$ is regular only when $\bm{A}$ is column regular. 

As usual, a vector $\bm{y}$ is linearly dependent on vectors $\bm{x}_{1},\ldots,\bm{x}_{m}$ if there are scalars $c_{1},\ldots,c_{m}\in\mathbb{X}$ such that $\bm{y}=c_{1}\bm{x}_{1}\oplus\cdots\oplus c_{m}\bm{x}_{m}$. Specifically, a vector $\bm{y}$ is collinear with $\bm{x}$ when $\bm{y}=c\bm{x}$ for some scalar $c$.

For any nonzero vector $\bm{x}=(x_{i})\in\mathbb{X}^{n}$, we introduce the multiplicative conjugate transpose to be a row vector $\bm{x}^{-}=(x_{i}^{-})$ with components $x_{i}^{-}=x_{i}^{-1}$ if $x_{i}\ne\mathbb{0}$, and $x_{i}^{-}=\mathbb{0}$ otherwise. The following properties of the conjugate transposition are easy to verify.

For any regular vectors $\bm{x}$ and $\bm{y}$ of the same size, the component-wise inequality $\bm{x}\leq\bm{y}$ implies that $\bm{x}^{-}\geq\bm{y}^{-}$ and vice versa. 

For any nonzero column vector $\bm{x}$, we have $\bm{x}^{-}\bm{x}=\mathbb{1}$. Moreover, if the vector $\bm{x}$ is regular, then $\bm{x}\bm{x}^{-}\geq\bm{I}$.

\subsection{Solution to Linear Inequality}
Given a matrix $\bm{A}\in\mathbb{X}^{n\times m}$, consider a problem that is to find regular vectors $\bm{x}\in\mathbb{X}^{n}$ to satisfy the inequality
\begin{equation}
\bm{A}\bm{x}
\leq
\bm{x}.
\label{I-Axx}
\end{equation}

Below, we present solutions to the inequality, which are obtained in \cite{Krivulin2006Solution,Krivulin2009Methods} and written here in a more compact equivalent form. 

For each matrix $\bm{A}\in\mathbb{X}^{n\times n}$, we introduce a function
$$
\mathop\mathrm{Tr}(\bm{A})
=
\mathop\mathrm{tr}\bm{A}\oplus\cdots\oplus\mathop\mathrm{tr}\bm{A}^{n}.
$$

If $\mathop\mathrm{Tr}(\bm{A})\leq\mathbb{1}$, we use a star operator that sends $\bm{A}$ to the matrix
$$
\bm{A}^{\ast}
=
\bm{I}\oplus\bm{A}\oplus\cdots\oplus\bm{A}^{n-1}.
$$

\begin{lemma}\label{L-I-Axx}
Let $\bm{x}$ be the complete regular solution to inequality \eqref{I-Axx}. Then the following statements hold:
\begin{enumerate}
\item If $\mathop\mathrm{Tr}(\bm{A})\leq\mathbb{1}$, then $\bm{x}=\bm{A}^{\ast}\bm{u}$ for all regular vectors $\bm{u}$.
\item If $\mathop\mathrm{Tr}(\bm{A})>\mathbb{1}$, then there is no regular solution.
\end{enumerate}
\end{lemma}

\section{Optimization Problem}\label{S-OP}

We now present the main result that solves an extended problem formulated in terms of a general idempotent semifield. We follow the solution approach, which is based on the derivation of sharp bounds on the objective function and applied to tropical optimization problems in a range of studies \cite{Krivulin2005Evaluation,Krivulin2009Methods,Krivulin2012Anew}.

Given matrices $\bm{A},\bm{B}\in\mathbb{X}^{m\times n}$ and vectors $\bm{p},\bm{q}\in\mathbb{X}^{m}$, the problem is to find regular vectors $\bm{x}\in\mathbb{X}^{n}$ such that
\begin{equation}
\begin{aligned}
&
\text{minimize}
&&
\bm{q}^{-}\bm{B}\bm{x}(\bm{A}\bm{x})^{-}\bm{p}.
\end{aligned}
\label{P-minqBxAxp}
\end{equation}

The following statement offers a direct solution to the problem.

\begin{theorem}\label{T-minqBxAxp}
Suppose that $\bm{A}$ is row regular and $\bm{B}$ is column regular matrices, $\bm{p}$ is nonzero and $\bm{q}$ is regular vectors. Denote $\Delta=(\bm{A}(\bm{q}^{-}\bm{B})^{-})^{-}\bm{p}$.

Then the minimum in problem \eqref{P-minqBxAxp} is equal to $\Delta$ and attained at any vector
$$
\bm{x}
=
\alpha(\bm{q}^{-}\bm{B})^{-},
\qquad
\alpha>\mathbb{0}.
$$
\end{theorem}
\begin{proof}
To verify the statement, we first show that $\Delta$ is a lower bound for the objective function in \eqref{P-minqBxAxp}, and then present vectors $\bm{x}$ that provide the bound.

Using the inequality $\bm{x}\bm{x}^{-}\geq\bm{I}$, we write
$$
\bm{q}^{-}\bm{B}\bm{x}\bm{x}^{-}
\geq
\bm{q}^{-}\bm{B}.
$$

Since the vector $\bm{q}$ is regular and the matrix $\bm{B}$ is column regular, the left and right sides of the last inequality are also regular. Furthermore, for any regular $\bm{x}$, we have $\bm{q}^{-}\bm{B}\bm{x}>\mathbb{0}$ and then write
$$
(\bm{q}^{-}\bm{B}\bm{x})^{-1}\bm{x}
=
(\bm{q}^{-}\bm{B}\bm{x}\bm{x}^{-})^{-}
\leq
(\bm{q}^{-}\bm{B})^{-}.
$$

Multiplication by $\bm{A}$ from the left gives
$$
(\bm{q}^{-}\bm{B}\bm{x})^{-1}\bm{A}\bm{x}
\leq
\bm{A}(\bm{q}^{-}\bm{B})^{-}.
$$

Considering that the matrix $\bm{A}$ is row regular, both sides of the inequality are regular vectors, and thus
$$
\bm{q}^{-}\bm{B}\bm{x}(\bm{A}\bm{x})^{-}
\geq
(\bm{A}(\bm{q}^{-}\bm{B})^{-})^{-}.
$$

After right multiplication of both sides by the vector $\bm{p}$, we finally have the lower bound in the form
$$
\bm{q}^{-}\bm{B}\bm{x}(\bm{A}\bm{x})^{-}\bm{p}
\geq
(\bm{A}(\bm{q}^{-}\bm{B})^{-})^{-}\bm{p}
=
\Delta
>
\mathbb{0}.
$$

It remains to verify that $\bm{x}=\alpha(\bm{q}^{-}\bm{B})^{-}$ yields the bound for any $\alpha>\mathbb{0}$. Indeed, substitution into the objective function and identity $\bm{x}^{-}\bm{x}=\mathbb{1}$ give
\begin{align*}
\bm{q}^{-}\bm{B}\bm{x}(\bm{A}\bm{x})^{-}\bm{p}
&=
\bm{q}^{-}\bm{B}(\bm{q}^{-}\bm{B})^{-}(\bm{A}(\bm{q}^{-}\bm{B})^{-})^{-}\bm{p}
\\
&=
(\bm{A}(\bm{q}^{-}\bm{B})^{-})^{-}\bm{p}
=
\Delta.
\hspace{4em}\qedhere
\end{align*}
\end{proof}

We conclude this section with solutions of two particular cases of problem \eqref{P-minqBxAxp}. First, assume that $\bm{B}=\bm{A}=\bm{I}$ and $\bm{p}=\bm{q}=\mathbb{1}$, where $\mathbb{1}$ denotes a vector that has all components equal to $\mathbb{1}$. We arrive at a problem in the form
\begin{equation*}
\begin{aligned}
&
\text{minimize}
&&
\mathbb{1}^{T}\bm{x}\bm{x}^{-}\mathbb{1}.
\end{aligned}
\end{equation*}

Application of Theorem~\ref{T-minqBxAxp} immediately gives $\Delta=\mathbb{1}$ as the minimum in the problem, which is attained at any vector $\bm{x}=\alpha\mathbb{1}$ for all $\alpha>\mathbb{0}$.

Finally, we examine a problem that underlies the design of optimal schedules to be given below. We put $\bm{A}=\bm{B}$, $\bm{p}=\bm{q}=\mathbb{1}$, and consider the problem
\begin{equation}
\begin{aligned}
&
\text{minimize}
&&
\mathbb{1}^{T}\bm{A}\bm{x}(\bm{A}\bm{x})^{-}\mathbb{1}.
\end{aligned}
\label{P-min1AxAx1}
\end{equation}

Using Theorem~\ref{T-minqBxAxp}, we readily obtain the following result.
\begin{corollary}\label{C-min1AxAx1}
Suppose that $\bm{A}$ is a regular matrix and denote $\Delta=(\bm{A}(\mathbb{1}^{T}\bm{A})^{-})^{-}\mathbb{1}$.

Then the minimum in problem \eqref{P-min1AxAx1} is equal to $\Delta$ and attained at any vector
$$
\bm{x}
=
\alpha(\mathbb{1}^{T}\bm{A})^{-},
\qquad
\alpha>\mathbb{0}.
$$
\end{corollary}

\section{Optimal Scheduling Problem}\label{S-AOS}

We are now in a position to place the scheduling problem described above into the framework of tropical mathematics and to give a direct solution to the problem in a compact vector form. 

\subsection{Representation of Scheduling Problem}

Consider problem \eqref{P-yyxayxbxI} and note that, in ordinary notation, it involves only the operations $\max$, addition, and additive inversion. Therefore, we can represent the problem in terms of the semifield $\mathbb{R}_{\max,+}$.

First, we write constraints as scalar equalities and inequalities:
\begin{align*}
\bigoplus_{j=1}^{n}c_{ij}x_{j}
&=
y_{i},
\\
\bigoplus_{j=1}^{n}d_{ij}x_{j}
&\leq
x_{i},
\qquad
i=1,\ldots,n.
\end{align*}

Using the matrices
$$
\bm{C}
=
(c_{ij}),
\qquad
\bm{D}
=
(d_{ij}),
$$
and the vectors
$$
\bm{x}
=
(x_{i}),
\qquad
\bm{y}
=
(y_{i}),
$$
the scalar constraints take the form
\begin{align*}
\bm{C}\bm{x}
&=
\bm{y},
\\
\bm{D}\bm{x}
&\leq
\bm{x}.
\end{align*}

Furthermore, we rewrite the objective function in \eqref{P-yyxayxbxI}. Since, for $\mathbb{R}_{\max,+}$, we have $\mathbb{1}=(0,\ldots,0)^{T}$, the objective function can be readily given by
$$
\left(\bigoplus_{i=1}^{n}y_{i}\right)\left(\bigoplus_{i=1}^{n}y_{i}^{-1}\right)
=
\mathbb{1}^{T}\bm{y}\bm{y}^{-}\mathbb{1}.
$$ 

Finally, by combining the objective function with the constraints, we arrive at the problem formulated in terms of $\mathbb{R}_{\max,+}$ to find vectors $\bm{x}$ and $\bm{y}$ such that
\begin{equation}
\begin{aligned}
&
\text{minimize}
&&
\mathbb{1}^{T}\bm{y}\bm{y}^{-}\mathbb{1},
\\
&
\text{subject to}
&&
\bm{C}\bm{x}
=
\bm{y},
\\
&
&&
\bm{D}\bm{x}
\leq
\bm{x}.
\end{aligned}
\label{P-1yy1CxyDxxI}
\end{equation}

\subsection{Solution to Scheduling Problem}

Under general conditions, a direct solution to \eqref{P-1yy1CxyDxxI} is obtained as follows.

\begin{theorem}\label{T-1yy1CxyDxxI}
Suppose that $\bm{C}$ is a regular matrix and $\bm{D}$ is a matrix that satisfies the condition $\mathop\mathrm{Tr}(\bm{D})\leq\mathbb{1}$. Denote $\Delta=(\bm{C}\bm{D}^{\ast}(\mathbb{1}^{T}\bm{C}\bm{D}^{\ast})^{-})^{-}\mathbb{1}$.

Then the minimum in problem \eqref{P-1yy1CxyDxxI} is equal to $\Delta$ and attained at
\begin{align*}
\bm{x}
&=
\alpha\bm{D}^{\ast}(\mathbb{1}^{T}\bm{C}\bm{D}^{\ast})^{-},
\\
\bm{y}
&=
\alpha\bm{C}\bm{D}^{\ast}(\mathbb{1}^{T}\bm{C}\bm{D}^{\ast})^{-}
\end{align*}
for all real numbers $\alpha$.
\end{theorem}
\begin{proof}
It follows from Lemma~\ref{L-I-Axx} that the inequality constraints in problem \eqref{P-1yy1CxyDxxI} have the solution $\bm{x}=\bm{D}^{\ast}\bm{u}$ for all regular vectors $\bm{u}$.

Based on the solution, the equality constraints become $\bm{y}=\bm{C}\bm{D}^{\ast}\bm{u}$.

Substitution of $\bm{y}$ in the objective function at \eqref{P-1yy1CxyDxxI} leads to problem \eqref{P-minqBxAxp} with $\bm{B}=\bm{A}=\bm{C}\bm{D}^{\ast}$, $\bm{p}=\bm{q}=\mathbb{1}$, and an unknown regular vector $\bm{u}$.

The obvious inequality $\bm{D}^{\ast}\geq\bm{I}$ implies that the matrix $\bm{C}\bm{D}^{\ast}$ is regular. The application of Corollary~\ref{C-min1AxAx1} gives the minimum $\Delta=(\bm{C}\bm{D}^{\ast}(\mathbb{1}^{T}\bm{C}\bm{D}^{\ast})^{-})^{-}\mathbb{1}$, which is attained at the vector $\bm{u}=\alpha(\mathbb{1}^{T}\bm{C}\bm{D}^{\ast})^{-}$. Back substitution of $\bm{u}$ leads to the desired solutions $\bm{x}=\alpha\bm{D}^{\ast}(\mathbb{1}^{T}\bm{C}\bm{D}^{\ast})^{-}$ and $\bm{y}=\alpha\bm{C}\bm{D}^{\ast}(\mathbb{1}^{T}\bm{C}\bm{D}^{\ast})^{-}$ for all $\alpha>\mathbb{0}$.
\end{proof}

Finally, we consider the problem, which was examined in \cite{Butkovic2009Onsome,Kin2010Optimizing} and can be now represented as
\begin{equation}
\begin{aligned}
&
\text{minimize}
&&
\mathbb{1}^{T}\bm{y}\bm{y}^{-}\mathbb{1},
\\
&
\text{subject to}
&&
\bm{C}\bm{x}
=
\bm{y}.
\end{aligned}
\label{P-1yy1Cxy}
\end{equation}

As a consequence of the solution to \eqref{P-1yy1CxyDxxI}, we get the following result.
\begin{corollary}\label{C-1yy1Cxy}
Suppose that $\bm{C}$ is a regular matrix and denote $\Delta=(\bm{C}(\mathbb{1}^{T}\bm{C})^{-})^{-}\mathbb{1}$.

Then the minimum in problem \eqref{P-1yy1Cxy} is equal to $\Delta$ and attained at
\begin{align*}
\bm{x}
&=
\alpha(\mathbb{1}^{T}\bm{C})^{-},
\\
\bm{y}
&=
\alpha\bm{C}(\mathbb{1}^{T}\bm{C})^{-}
\end{align*}
for any real number $\alpha$.

\end{corollary}

Note that the solutions to the above problems are given up to a scale factor $\alpha$. In the context of scheduling, this form of solutions offers a room to accommodate additional constraints such as a due date for the project.

\subsection{Numerical Examples}

To illustrate the results obtained, we consider an example project of three activities under constraints given by the matrices
$$
\bm{C}
=
\left(
\begin{array}{ccc}
4 & 0 & \mathbb{0}
\\
2 & 3 & 1
\\
1 & 1 & 3
\end{array}
\right),
\qquad
\bm{D}
=
\left(
\begin{array}{rrr}
\mathbb{0} & -2 & 1
\\
0 & \mathbb{0} & 2
\\
-1 & \mathbb{0} & \mathbb{0}
\end{array}
\right),
$$
where the symbol $\mathbb{0}=-\infty$ is used for ease of exposition.

First, we do not take into account the start-start constraints to solve the reduced problem \eqref{P-1yy1Cxy}. After calculating the vectors
$$
(\mathbb{1}^{T}\bm{C})^{-}
=
\left(
\begin{array}{r}
-4
\\
-3
\\
-3
\end{array}
\right),
\qquad
\bm{C}(\mathbb{1}^{T}\bm{C})^{-}
=
\left(
\begin{array}{r}
0
\\
0
\\
0
\end{array}
\right),
$$
we apply Corollary~\ref{C-1yy1Cxy} and immediately arrive at the solution
$$
\Delta
=
0,
\qquad
\bm{x}
=
\alpha
\left(
\begin{array}{r}
-4
\\
-3
\\
-3
\end{array}
\right),
\qquad
\bm{y}
=
\alpha
\left(
\begin{array}{r}
0
\\
0
\\
0
\end{array}
\right),
$$
where $\alpha$ is any number such that $\alpha>\mathbb{0}=-\infty$.

Note that, in this situation, we really get a just-in-time schedule with a single common completion time of all activities.

Let us now incorporate the start-start constraints given by $\bm{D}$ into the problem. We take $\bm{D}$ to calculate
$$
\bm{D}^{2}
=
\left(
\begin{array}{rrr}
0 & \mathbb{0} & 0
\\
1 & -2 & 1
\\
\mathbb{0} & -3 & 0
\end{array}
\right).
$$

Furthermore, we obtain
$$
\bm{D}^{3}
=
\left(
\begin{array}{rrr}
-1 & -2 & 1
\\
0 & -1 & 2
\\
-1 & \mathbb{0} & -1
\end{array}
\right),
\qquad
\mathop\mathrm{Tr}(\bm{D})
=
0,
$$
and then evaluate the sum
$$
\bm{D}^{\ast}
=
\bm{I}
\oplus
\bm{D}
\oplus
\bm{D}^{2}
=
\left(
\begin{array}{rrr}
0 & -2 & 1
\\
1 & 0 & 2
\\
-1 & -3 & 0
\end{array}
\right).
$$

Since the first and the third columns in the matrix $\bm{D}^{\ast}$ are collinear, we can drop the last column to simplify the solution. We successively get the matrices
$$
\bm{D}^{\ast}
=
\left(
\begin{array}{rr}
0 & -2 
\\
1 & 0
\\
-1 & -3
\end{array}
\right),
\qquad
\bm{C}\bm{D}^{\ast}
=
\left(
\begin{array}{cc}
4 & 2
\\
4 & 3
\\
2 & 1
\end{array}
\right),
$$
and calculate the vector
\begin{align*}
(\mathbb{1}^{T}\bm{C}\bm{D}^{\ast})^{-}
&=
\left(
\begin{array}{r}
-4
\\
-3
\end{array}
\right),
\\
\bm{D}^{\ast}(\mathbb{1}^{T}\bm{C}\bm{D}^{\ast})^{-}
&=
\left(
\begin{array}{r}
-4
\\
-3
\\
-5
\end{array}
\right),
\\
\bm{C}\bm{D}^{\ast}(\mathbb{1}^{T}\bm{C}\bm{D}^{\ast})^{-}
&=
\left(
\begin{array}{r}
0
\\
0
\\
-2
\end{array}
\right).
\end{align*}

The application of Theorem~\ref{T-1yy1CxyDxxI} gives the results
$$
\Delta
=
2,
\quad
\bm{x}
=
\alpha
\left(
\begin{array}{r}
-4
\\
-3
\\
-5
\end{array}
\right),
\quad
\bm{y}
=
\alpha
\left(
\begin{array}{r}
0
\\
0
\\
-2
\end{array}
\right),
$$
where $\alpha$ is any real number.

The solution offers a schedule that is optimal with respect to the span seminorm. Note, however, that the constraints in the problem give no way for the schedule to provide a single common completion time of all activities.


\bibliographystyle{utphys}

\bibliography{Explicit_solution_of_a_tropical_optimization_problem_with_application_to_project_scheduling}

\end{document}